\documentclass{amsart}

  \usepackage[utf8]{inputenc}
  \usepackage{amsmath}
  \usepackage{amssymb}
  \usepackage{amsthm}
  \usepackage{mathrsfs}
  \usepackage{bm}
  \usepackage{bbm}
  \usepackage{graphicx}
  \usepackage{color}
  \usepackage{enumerate}
	\newtheoremstyle{slanted}
	{}
	{}
	{\slshape}
	{}
	{\bfseries}
	{.}
	{ }
	{}
	
	\theoremstyle{slanted}
	\newtheorem{theo}{Theorem}[section]
	\newtheorem{prop}[theo]{Proposition}
	\newtheorem{conjecture}[theo]{Conjecture}
	\newtheorem{question}[theo]{Question}
	\newtheorem{lemma}[theo]{Lemma}
	\newtheorem{definition}[theo]{Definition}
	\newtheorem{corollary}[theo]{Corollary}

	\def\egdef{:=}

	\newcommand{\stack}[2]{\array{c}{\scriptstyle #1}\\[-1.1ex]{\scriptstyle #2}\endarray}
	\newcommand{\tend}[3][]{\xrightarrow[#2\to#3]{#1}}
		
	\newcommand{\EE}{\mathbb{E}}

	\def\ind#1{\mathbbmss{1}_{#1}}
	\newcommand{\ZZ}{\mathbb{Z}}

	\newcommand{\NN}{\mathbb{N}}
	\newcommand{\NNS}{\mathbb{N^*}}
	\newcommand{\PP}{\mathbb{P}}

	\newcommand{\B}{\mathscr{B}}

	\newcommand{\mob}{\boldsymbol{\mu}}
	\newcommand{\om}{\underline{\omega} }
	\newcommand{\nub}{\nu_\B}
	\newcommand{\nup}{\nu'}

	\title[Dynamical Point of View on $B$-free Numbers]{A Dynamical Point of View on the Set of $\B$-free Integers}
	\author{El Houcein El Abdalaoui}
	\author{Mariusz Lema\'nczyk}
	\author{Thierry de la Rue}
	\address{El Houcein El Abdalaoui, Thierry de la Rue:
	Laboratoire de Math\'ematiques Rapha\"el Salem,
	Normandie Université, Universit\'e de Rouen, CNRS --
	Avenue de l'Universit\'e --
	76801 Saint \'Etienne du Rouvray, France.}
	\email{elhoucein.elabdalaoui@univ-rouen.fr\\Thierry.de-la-Rue@univ-rouen.fr}
	\address{Mariusz Lema\'nczyk: Faculty of Mathematics and Computer Science, Nicolaus Copernicus University, 12/18 Chopin street, 87-100 Toru\'{n}, Poland}
	\email{mlem@mat.umk.pl}
	\thanks{Research supported by Science National Center (NCN) grant DEC-2011/03/B/ST1/00407}
	\thanks{The results of the paper have been obtained during the visit of the second author at Rouen University in September 2013.}

\begin{document}

\bibliographystyle{amsplain}

\begin{abstract}
  Sarnak has recently initiated the study of the Möbius function and its square, the characteristic function of square-free integers, from a dynamical point of view, introducing the Möbius flow and the square-free flow as the action of the shift map on the respective subshfits generated by these functions. 
  
  In this paper, we extend the study of the square-free flow to the more general context of $\B$-free integers, that is to say integers with no factor in a given family $\B$ of pairwise relatively prime integers, the sum of whose reciprocals is finite. Relying on dynamical arguments, we prove in particular that the distribution of patterns in the characteristic function of the $\B$-free integers follows a shift-invariant probability measure, and gives rise to a measurable dynamical system isomorphic to a specific minimal rotation on a compact group. 
  As a by-product, we get the abundance of twin $\B$-free numbers. Moreover, we show that the distribution of patterns in small intervals of the form $[N,N+\sqrt{N})$ also conforms to the same measure.
  
  When elements of $\B$ are squares, we introduce a generalization of the Möbius function, and discuss a conjecture of Chowla in this broader context. 
  \end{abstract}

\maketitle

\section{Introduction and overview of the paper}

Throughout the paper, the symbol $\NNS$ denotes the set of natural numbers starting from 1: $\NNS\egdef\{1,2,3,\ldots\}$. For any set $A$, we denote by $|A|$ the cardinality of $A$.

We are interested in arithmetic functions defined on $\NN^*$ and taking values in a subset $\Gamma\subset\{-1,0,1\}$. We view these functions as points $x=(x_n)_n\in\NNS$ in the compact topological space $\Gamma^\NNS$. On this space, we consider the \emph{shift map} $S: x=(x_n)_{n\in\NNS}\mapsto Sx=(x'_n)_{n\in\NNS}$, where for each $n\in\NNS$, $x'_n\egdef x_{n+1}$. 

We say that $x\in \Gamma^\NNS$ is \emph{generic} for some probability measure $\nu$ on $\Gamma^\NNS$ if the frequency of patterns in the sequence $x$ conforms to $\nu$. Formally, this means that for each cylinder set $C\subset \Gamma^\NNS$, we have
\[
 \dfrac{1}{N} \sum_{0\le n<N} \ind{C}(S^nx)\tend{N}{\infty} \nu(C). 
\]
Observe that if $x$ is generic for $\nu$, then $\nu$ must be shift-invariant. 

\subsection{The Möbius flow and the square-free flow}

Recall that the \emph{Möbius function} is the multiplicative arithmetic function $\mob$ defined, for all positive integer $n$, by
\[
  \mob(n) \egdef \begin{cases}
              0 \text{ if $n$ is divisible by the square of a prime number,}\\
              (-1)^{d_n} \text{ otherwise, where $d_n$ is the number of prime factors of $n$.}
            \end{cases}
\]
Its square $\mob^2$ is the characteristic function of the square-free integers.
In spite of their seemingly artificial definitions, these functions are of great importance in number theory for their connections with the distribution of prime numbers and the Riemann zeta function (see \textit{e.g.} \cite{apostol,hildebrand,titchmarsh}). 

Recently, Sarnak~\cite{sarnak} has suggested to study $\mob$ and $\mob^2$ from a dynamical point of view: He introduced the \emph{Möbius flow}  (respectively the \emph{square-free flow}) as the action of the shift map on the closure of the orbit of $\mob$ in $\{-1,0,1\}^\NNS$ (respectively of $\mob^2$ in  $\{0,1\}^\NNS$). Sarnak announced several results concerning the square-free flow: Its topological entropy is $6/\pi^2$ (when the logarithm is taken in base~2), and the sequence $\mob^2$ is generic for a shift-invariant probability measure $\nu_S$ on $\{0,1\}^\NNS$ which has zero Kolmogorov entropy. He also stated important conjectures on the Möbius flow: In particular, he explained that the sequence $\mob$ is expected to be generic for a probability measure $\nu_N$ which can be seen as the ``fully random'' extension of $\nu_S$, and connects this conjecture to a famous conjecture of Chowla~\cite{chowla}. 

\smallskip

Our purpose in the present work is to extend this point of view to a generalized setting, providing simple dynamical proofs of several known results concerning square-free numbers, and extending them to the context of so-called \emph{$\B$-free numbers} (see below). We also consider a generalization of the Möbius function, which we hope can shed a new light upon the Chowla conjecture.

\subsection{$\B$-free numbers}

More generally than the square-free numbers, we consider all natural numbers which have no factors in a given subset $\B=\{b_k:k\ge1\}\subset\{2,3,4,\ldots\}$ satisfying
\begin{equation}
  \label{eq:relatively prime}
  \forall 1\le k< k',\ \text{$b_k$ and $b_{k'}$ are relatively prime,}
\end{equation}
and
\begin{equation}
  \label{eq:finite_sum_of_inverse}
  \sum_{k\ge 1} \dfrac{1}{b_k} < \infty.
\end{equation}
Integers with no factors in $\B$ are called \emph{$\B$-free}. Those integers have been introduced by Erdös~\cite{erdos}, who proved the existence of some $0<c<1$ such that, for all large enough $N$, the interval $[N,N+N^c)$ contains at least one $\B$-free number. A special case of $\B$-free numbers, in which falls the case of square-free numbers, had been previously considered by Mirsky~\cite{mirsky}, who studied the distribution of patterns in the characteristic function of \emph{$r$-free numbers}, that is to say numbers which are not divisible by the $r$-th power of any prime ($r\ge2$). 

We introduce the sequence $\eta=(\eta_n)_{n\in\NNS}$ which is the characteristic function of the set of $\B$-free integers:
\begin{equation}
  \label{eq:def_eta}
  \eta_n\egdef \begin{cases}
                  0 \text{ if there exists $k\ge1$ such that $b_k$ divides $n$,}\\
                  1 \text{ otherwise, \textit{i.e.} if $n$ is $\B$-free.}
                \end{cases}
\end{equation}

We define a subshift $X_\B\subset\{0,1\}^{\NNS}$ (Section~\ref{sec:B-admissible}), and a shift-invariant probability measure $\nub$ on $X_\B$ (Section~\ref{sec:nub}) such that:
\begin{itemize}
  \item The closure of the orbit of $\eta$ under the shift map is $X_\B$ (Corollary~\ref{cor:subshift}).
  \item The topological entropy of the subshift $X_\B$ is positive, equal to $\prod_{k=1}^\infty (1 - 1/b_k)$ (Theorem~\ref{thm:entropy}). This generalizes the formula given by Sarnak~\cite{sarnak} in the context of the square-free flow, a proof of which can be found in~\cite{peckner}.
  \item The sequence $\eta$ is $\nub$-generic (Theorem~\ref{thm:eta_generic}).
  \item The dynamical system $(X_\B,\nub,S)$ has zero Kolmogorov entropy, even discrete spectrum. In fact, we provide an explicit measurable isomorphism between this system and the minimal rotation $T$, which is the addition of $(1,1,\ldots)$ on the compact abelian group $\Omega\egdef\prod_{k\ge1}\ZZ/b_k\ZZ$ (Theorem~\ref{thm:isomorphism}).
\end{itemize}
The last point generalizes the main result of Cellarosi and Sinai~\cite{cs}, concerning the square-free flow, to the context of $\B$-free numbers. But we point out that, while these authors make use of Mirsky's theorem providing the frequency of patterns in the sequence $\mob^2$, our Theorem~\ref{thm:eta_generic} provides an alternative, purely dynamic, proof of Mirsky's result. We also show that the method to prove Theorem~\ref{thm:eta_generic} can be adapted to study the frequency of patterns in $\eta$ along arithmetic subsequences (Theorem~\ref{thm:arithmetic-subsequences}). Moreover, provided an extra condition on $\B$ which is automatically satisfied in the case of $r$-free numbers ($r\ge2$), we prove that the frequency of patterns in $\eta$ in ``short intervals'' of the form $[N, N+\sqrt{N})$ is also given by the measure $\nub$ (Theorem~\ref{thm:short_intervals}). 
The existence of $\B$-free numbers in intervals of the form $[N,N+N^c)$ has been very much studied since the work of Erdös~\cite{erdos}, in which it is conjectured that for any $c>0$, the interval  $[N,N+N^c)$ always contains at least one $\B$-free number for $N$ large enough (see~\cite{wu}, \cite{az} and references therein). 
But, to our knowledge, this is the first time that frequency of patterns inside short intervals is considered.

\subsection{A generalization of the Möbius function}

In the last section of the present paper, we propose to study a generalization of the Möbius function. We assume here that all integers $b_k\in\B$ are squares: $b_k=a_k^2$, where the integers $a_k$ are pairwise relatively prime. For each positive integer $n$, we define $\delta_n$ as the number of integers $a_k$ dividing $n$, and consider the sequence $\pi=(\pi_n)_{n\in\NNS}$ defined by $\pi_n\egdef (-1)^{\delta_n}$. Finally, we define $\mu=(\mu_n)_{n\in\NNS}$, where $\mu_n\egdef \eta_n\cdot\pi_n$. Observe that, if the set $\{a_k: k\ge1\}$ is the set of prime numbers, we recover in this way the Möbius function: $\mu_n=\mob(n)$. (We therefore refer to this situation as the \emph{Möbius case}.) We introduce an abstract property of the sequence $\mu$, called the \emph{Chowla property}, which in the Möbius case corresponds exactly to the Chowla conjecture. We show that this Chowla property is equivalent to the genericity of $\mu$ for a shift-invariant probability measure $\nu_M$ on $\{-1,0,1\}^\NNS$, which is the 
``completely random'' extension 
of $\nu_\B$ (Theorem~\ref{thm:chowla and genericity}). Moreover, this property is satisfied as soon as $\pi$ is generic for the Bernoulli measure $\beta\egdef (1/2,1/2)^{\otimes\NNS}$ on $\{-1,1\}^\NNS$.

An important parameter in our study is $\Sigma\egdef\sum_{k\ge1}1/a_k$. When $\Sigma<\infty$, we prove that $\pi$ is indeed generic, however the shift-invariant measure $\nu'$ which is obtained has zero Kolmogorov entropy (Theorem~\ref{thm:pi_generic}). In the Möbius case, we have $\Sigma=\infty$, and we are not able to say whether $\pi$ is generic for some measure or not. However, we show that, as $\Sigma\to\infty$, the probability measure $\nu'$ for which $\pi$ is generic converges weakly to the Bernoulli measure $\beta$ (Theorem~\ref{thm:Sigma_to_infty}).

\section{The $\B$-free flow}

In this section, we introduce the study of the sequence $\eta$, which is the characteristic function of the set of $\B$-free integers, from a dynamical point of view. We work here in the space $X\egdef\{0,1\}^\NNS$. 

\subsection{A dynamical system which outputs $\eta$}

The starting point of our argument is the following simple fact:
We can observe $\eta$ along the orbit of a particular point in a specific dynamical system. 
Indeed, consider the compact additive group 
\[
  \Omega \egdef \prod_{k\ge 1} \ZZ / b_k\ZZ,
\]
and the transformation $T$ of $\Omega$ mapping $\omega=(\omega_k)_{k\ge1}$ to $T\omega\egdef(\omega_k+1)_{k\ge1}$. 
Then we can get our sequence $\eta$ by observing the function $f:\ \Omega\to\{0,1\}$ defined by
\[
  f(\omega)\egdef \begin{cases}
                     0 \text{ if there exists $k\ge1$ such that $\omega_k=0$,}\\
                     1 \text{ otherwise.}
                  \end{cases}
\]
along the orbit of the element $\om\egdef (0,0,0,\ldots)\in\Omega$:
\begin{equation}
  \label{eq:output_eta}
  \forall n\in\NNS,\quad \eta_n = f(T^n\om).
\end{equation}

More generally, it will be useful to introduce the map $\varphi:\ \Omega\to X$ defined by
\begin{equation}
  \label{eq:def_phi}
  \varphi(\omega) \egdef \Bigl( f(T^n\omega)\Bigr)_{n\in\NNS},
\end{equation}
so that $\eta=\varphi(\om)$.
Note that $\varphi(\omega)$ is the sequence $(x_n)_{n\in\NNS}\in X$ satisfying
\begin{equation}
 \label{eq:im_phi}
 \forall n\in\NNS,\ x_n=\begin{cases}
                            0\text{ if }\exists k\ge1:\ \omega_k+n=0\ [b_k],\\
                            1\text{ if }\forall k\ge1:\ \omega_k+n\neq 0\ [b_k].
                         \end{cases}
\end{equation}

By construction, we have
$S\circ \varphi = \varphi\circ T$ (indeed,
$(S\varphi(\omega))_n=\varphi(\omega)_{n+1}=(\varphi(T\omega))_n$). However, $\varphi$ is not a topological factor map because it is not continuous
(\textit{e.g.}, if $2\notin\B$,
$\varphi(1,1,\ldots,1,1,\ldots)_1=
1\neq0=\varphi(1,1,\ldots,1,-1,-1,\ldots)_1$).

\subsection{Admissible sequences}
\label{sec:B-admissible}
Let $A$ be a subset of $\NNS$. For each $b\ge 1$, we denote by $t(A,b)$ the number of classes modulo $b$ in $A$:
\[ t(A,b) \egdef \left|\bigl\{z\in\ZZ/b\ZZ: \exists n\in A, n=z\ [b]\bigr\}\right|. \]
\begin{definition}
A subset $A\subset\NNS$ is said to be \emph{$\B$-admissible}  if 
\begin{equation}
  \label{eq:def-admissible}
\forall k\ge 1,\ t(A,b_k)<b_k. 
\end{equation}

An infinite sequence $x=(x_n)_{n\in\NNS}\in X$ is said to be \emph{$\B$-admissible} if its \emph{support}
$\{n\in\NNS: x_n=1\}$ is $\B$-admissible. In the same way, a finite block $x_1\ldots x_N\in\{0,1\}^N$ is \emph{$\B$-admissible} if 
$\{n\in\{1,\ldots,N\}: x_n=1\}$ is $\B$-admissible.
\end{definition}

We denote by $X_\B$ the set of all $\B$-admissible sequences in $X$. Since a set
is $\B$-admissible if and only if each of its finite subsets is $\B$-admissible, and a translation of a $\B$-admissible set is $\B$-admissible,   $X_\B$ is a closed and shift-invariant subset of $X$, \textit{i.e.} a subshift.

\begin{prop}
\label{prop:admissible}
  For any $\omega\in\Omega$, $\varphi(\omega)$ is a $\B$-admissible sequence. In particular, $\eta$ is $\B$-admissible.
\end{prop}

\begin{proof}
  Let $\omega\in\Omega$, and denote by $A$ the support of $\varphi(\omega)$. Then for each $k\ge1$, we have $-\omega_k\not \in \{ n\mod b_k: n\in A\}$. 
\end{proof}

\subsection{A shift-invariant measure on $X_\B$}
\label{sec:nub}
Since the integers $b_k$ are pairwise relatively prime, the topological dynamical system $(\Omega,T)$ is a minimal rotation, hence also uniquely ergodic, and its unique ergodic invariant probability measure is the normalized Haar measure on $\Omega$, which we denote by $\PP$. Under $\PP$, the coordinates $\omega_k$, $k\ge 1$ are independent, each of them being uniformly distributed in the corresponding finite group $\ZZ/b_k\ZZ$. 

Let $\nub$ be the pushforward measure of $\PP$ to $X$ by the map $\varphi$: For each measurable subset $C\subset X$,
$\nub(C)\egdef \PP(\varphi^{-1}C)$. Then $\nub$ is shift-invariant, and by Proposition~\ref{prop:admissible}, $\nub$ is concentrated on $X_\B$. Moreover, the measurable dynamical system $(X_\B,\nub,S)$ is a factor of $(\Omega,\PP,T)$. In particular, it is ergodic.

We shall need in the sequel the following lemma, valid in the context of any probability space $(\Omega,\PP)$.

\begin{lemma}
\label{lemma:general}
  In the probability space $(\Omega,\PP)$, let $(E_n)_{n\in\NNS}$ be a sequence of events, and for each $n$, let $F_n\egdef \Omega\setminus E_n$. Then, for any finite disjoint subsets $A,B\subset\NNS$, we have
  \[ \PP  \left(   \bigcap_{n\in A}E_n  \cap \bigcap_{m\in B}F_m\right) = \sum_{A\subset D\subset A\cup B} (-1)^{|D\setminus A|}\ \PP  \left(   \bigcap_{d\in D}E_d \right)   . \]
\end{lemma}

\begin{proof}
Denoting by $\EE[\,\cdot\,]$ the expectation with respect to the probability $\PP$, we have
 \begin{align*}
  \PP  \left(   \bigcap_{n\in A}E_n  \cap \bigcap_{m\in B}F_m\right)
    & = \EE \left[ \prod_{n\in A}\ind{E_n} \prod_{m\in B}\left(1-\ind{E_m}\right) \right] \\
    & = \sum_{C\subset B} (-1)^{|C|}\ \EE \left[ \prod_{n\in A\cup C}\ind{E_n} \right] \\
    & = \sum_{A\subset D\subset A\cup B} (-1)^{|D\setminus A|}\ \PP  \left(   \bigcap_{d\in D}E_d \right).
\end{align*}
\end{proof}

Now we want to compute the values taken by $\nub$ on cylinder sets in $X$. For any finite disjoint subsets $A,B\subset \NNS$, we denote by $C_{A,B}$ the cylinder set
\[  C_{A,B} \egdef \left\{ x=(x_n)_{n\in\NNS}\in X: \forall n\in A, x_n=1, \forall m\in B, x_m=0 \right\}.\]
When $B$ is empty, we rather write $C_A^1$ for $C_{A,\emptyset}$. Analogously, we write $C_B^0$ for $C_{\emptyset,B}$.

\begin{prop}
  \label{prop:measure_of_cylinders}
  For any finite disjoint subsets $A,B\subset \NNS$, we have 
  \begin{equation}
    \label{eq:nus_CA}
    \nub(C_A^1) = \prod_{k\ge1}\left( 1 - \dfrac{t(A,b_k)}{b_k}\right),
  \end{equation}
  and
  \begin{equation}
    \label{eq:nus_CAB}
    \nub(C_{A,B}) = \sum_{A\subset D\subset A\cup B} (-1)^{|D\setminus A|} \prod_{k\ge1}\left( 1 - \dfrac{t(D,b_k)}{b_k}\right).
  \end{equation}
\end{prop}
\begin{proof}
Let $A$ be a finite subset of $\NNS$. 
First notice that for a fixed $k\geq 1$,
\[
\PP\bigl\{\omega\in\Omega: \forall n\in A,\ \omega_k+n\neq 0\, [b_k]\bigr\} = 1-\frac{t(A,b_k)}{b_k}.
\]
Then, we have
  \begin{align*}
    \nub(C_A^1) 
    & = \PP\left( \varphi^{-1}(C_A^1)\right)\\
    & = \PP\left( \bigcap_{k\ge1} \bigl\{\omega\in\Omega: \forall n\in A, \omega_k+n\neq 0 \, [b_k]\bigr\}\right)\\
    & = \prod_{k\ge 1} \left( 1 - \dfrac{t(A,b_k)}{b_k}\right),
  \end{align*}
which proves~\eqref{eq:nus_CA}. Now, observe that~\eqref{eq:nus_CAB} can be derived from~\eqref{eq:nus_CA} using Lemma~\ref{lemma:general}.
\end{proof}

\begin{prop}
  For  any finite disjoint subsets $A,B\subset \NNS$, the following properties are equivalent:
  \begin{enumerate}[(i)]
    \item $\nub(C_{A,B})>0$;
    \item $\nub(C_{A}^1)>0$;
    \item $A$ is a $\B$-admissible set.
  \end{enumerate}
\end{prop}

\begin{proof}
  By~\eqref{eq:nus_CA} and ~\eqref{eq:finite_sum_of_inverse}, (ii) is equivalent to $t(A,b_k)<b_k$ for all $k\ge 1$, which is the definition of a $\B$-admissible set. Obviously, (i) implies (ii), and it only remains to show that (ii)$\Longrightarrow$(i). 

Assume that $\nub(C_{A}^1)>0$.
  We want to prove that $\nub(C_{A,B})=\PP(\varphi^{-1}C_{A,B})>0$. 
  Since $A$ and $B$ are disjoint, and since $b_k\tend{k}{\infty}\infty$ by~\eqref{eq:finite_sum_of_inverse}, there exists $K$ such that, for all $j\ge 1$, the equation 
  $n=m\ [b_{K+j}]$ has no solution with $n\in A$ and $m\in B$. 
  Let $B=\{m_1,\ldots,m_r\}$. Observe that 
\begin{multline*}
    \bigl\{\omega\in\Omega:  \forall 1\le j\le r, \omega_{K+j} + m_j = 0\ [b_{K+j}]\bigr\} \\ 
     \cap \bigl\{\omega\in\Omega:  \forall k\notin\{K+1,\ldots,K+r\},  \forall n\in A, \omega_k+n\neq 0\ [b_k]\bigr\}
       \subset \varphi^{-1}(C_{A,B}).
   \end{multline*}
But the set on the left-hand side is the intersection of two independent events in $(\Omega,\PP)$, the first one has probability $\prod_{j=1}^r1/b_{K+j}>0$, and the second one contains $\varphi^{-1}(C_{A}^1)$, therefore has positive probability. 
\end{proof}

\begin{corollary}
 The support of $\nub$ is the subshift $X_\B$ of $\B$-admissible sequences.
\end{corollary}

\section{The measurable dynamical system  $(X_\B,\nub,S)$}

The purpose of this section is to prove the following result, which provides a complete description of the measurable dynamical system  $(X_\B,\nub,S)$.

\begin{theo}
\label{thm:isomorphism}
   The map $\varphi$ is an isomorphism between the two measurable dynamical systems $(\Omega,\PP,T)$ and $(X_\B,\nub,S)$.
\end{theo}

We already know that $\varphi$ is a measurable factor map from $(\Omega,\PP,T)$ to $(X_\B,\nub,S)$.  Therefore the proof of the theorem reduces to the proof of the following result.
\begin{prop}
  \label{prop:one-to-one}
  There exists $\Omega_0\subset\Omega$, with $\PP(\Omega_0)=1$, such that for each $\omega\in\Omega_0$, the preimage of $\varphi(\omega)$ is reduced to the singleton $\{\omega\}$.
\end{prop}

For $b\ge 1$ and $z\in\ZZ/b\ZZ$, let us denote by $\pi_{z,b}$ the infinite arithmetic subsequence 
\[ \pi_{z,b}\egdef\{n\in\NNS: z+n=0\ [b]\}. \]
Fix $k\ge1$. By construction of $\varphi$, the sequence $\varphi(\omega)$ is equal to zero along the infinite arithmetic subsequence $\pi_{\omega_k,b_k}$. The proof of Proposition~\ref{prop:one-to-one} will consist in showing that, with probability one,  there is no other arithmetic subsequence $\pi_{z,b_{k}}$ with $z\in\ZZ/b_k\ZZ$ along which $\varphi(\omega)$ is equal to zero. It will follow that, with probability one, we can recover the coordinate $\omega_k$ from $\varphi(\omega)$ by observing along which infinite arithmetic subsequence $\pi_{z,b_k}$ the sequence $\varphi(\omega)$ vanishes. 

\begin{lemma}
\label{lemma:invariant}
 For each $k\ge 1$, the $\sigma$-algebra of $T^{b_k}$-invariant sets in $(\Omega,\PP,T)$ coincides modulo $\PP$ with the $\sigma$-algebra generated by the $k$-th coordinate $\omega_k$.
\end{lemma}

\begin{proof}
Let us fix $k\ge1$, and consider for any $z\in\ZZ/b_k\ZZ$ the subset
$\Omega_{k,z}\egdef\{\omega\in\Omega:\omega_k=z\}$. Obviously $\Omega_{k,z}$ is invariant under the action of $T^{b_k}$. We also note that, for all $k'\neq k$, since  the integer $b_{k'}$ is relatively prime with $b_k$,  the addition of $b_k$ in $\ZZ/b_{k'}\ZZ$ is isomorphic to the addition of 1. Let $\tilde\Omega\egdef\prod_{k'\neq k'} \ZZ/b_{k'}\ZZ$,  denote by $\tilde T$ the addition of $(1,1,\ldots)$ in $\tilde\Omega$, and by $\tilde\PP$ the normalized Haar measure on $\tilde\Omega$. Then the dynamical system $\bigl(\Omega_{k,z},\PP(\,\cdot\,|\,\Omega_{k,z}),T^{b_k}\bigr)$ is measurably isomorphic to $(\tilde\Omega,\tilde\PP,\tilde T)$. Since the latter is an ergodic rotation, the action of $T^{b_k}$ on $\Omega_{k,z}$ is ergodic. Now, if $A\subset \Omega$ is a $T^{b_k}$-invariant set, by ergodicity, we must have $\PP(A\,|\, \Omega_{k,z})\in\{0,1\}$. Therefore $A$ coincides modulo $\PP$ with the union of those $\Omega_{k,z}$ satisfying $\PP(A\,|\, \Omega_{k,z})=1$.
\end{proof}

\begin{proof}[Proof of Proposition~\ref{prop:one-to-one}]
Fix $k\ge1$ and $z\in\ZZ/b_k\ZZ$. Since $\varphi\circ T^{b_k}=S^{b_k}\circ \varphi$, for any $z\in\ZZ/b_k\ZZ$, the event 
\[
E_{k,z}\egdef\{\omega\in\Omega:\varphi(\omega)=0\text{ along }\pi_{z,b_k}\} 
\]
is $T^{b_k}$-invariant, and contains $\Omega_{k,z}$. Hence, by Lemma~\ref{lemma:invariant}, it coincides $\PP$-almost surely with an $\omega_k$-measurable event, which is of the form $\bigl\{\omega\in\Omega:\omega_k\in\{z_1,\ldots,z_\ell\}\bigr\}$ for some finite subset $\{z_1,\ldots,z_\ell\}\subset \ZZ/b_k\ZZ$. We want to show that this finite subset is reduced to the singleton $\{z\}$. For this, it is enough to prove that, for any $z'\in\ZZ/b_k\ZZ\setminus\{z\}$, $\PP(E_{k,z}|\Omega_{k,z'})=\PP(E_{k,z}|\omega_k=z') <1$. 
Fix some $n\in\pi_{z,b_k}$, and note that $z'+n\neq 0\ [b_k]$. We have
\begin{align*}
 \PP(E_{k,z}|\omega_k=z') 
    & \le \PP\Bigl(\varphi(\omega)_n=0\,\bigr|\,\omega_k=z'\Bigr)\\
    & = \PP\Bigl(\exists k'\neq k: \omega_{k'}+n=0\ [b_{k'}]\,\bigr|\,\omega_k=z'\Bigr) \\
    & = \PP\bigl(\exists k'\neq k: \omega_{k'}+n=0\ [b_{k'}]\bigr) \quad\text{by independence in $(\Omega,\PP)$}\\
    & \le \PP\bigl(\varphi(\omega)_n=0\bigr) \\
    & = 1 - \nub(C^1_{\{n\}}) < 1 \quad\text{by~\eqref{eq:nus_CA}}.
\end{align*}

It follows that 
\[
  \PP\bigl( E_{k,z}\setminus\{\omega\in\Omega:\omega_k=z\} \bigr) = 0.
\]
We now get the proposition by setting
\[
  \Omega_0 \egdef \Omega \setminus \bigcup_{k\ge 1} \bigcup_{z\in\ZZ/b_k\ZZ}
  \bigl( E_{k,z}\setminus\{\omega\in\Omega:\omega_k=z\} \bigr).
  \]
Indeed, notice that for each $\omega\in\Omega$ and each $k\ge1$, we have $\omega\in E_{k,\omega_k}$.
Assume now that $\phi(\omega')=\phi(\omega'')$, but that $\omega'\neq\omega''$. Then there exists $k\ge1$ with
$\omega'_k\neq \omega''_k$. We have $\omega''\in E_{k,\omega''_k}$, but also $\omega'\in E_{k,\omega''_k}$ since $\phi(\omega')=\phi(\omega'')$. Thus $\omega'\in E_{k,\omega''_k}\setminus\{\omega: \omega_k=\omega''_k\}$ and therefore $\omega'\notin\Omega_0$.
\end{proof}

\section{Frequency of blocks in $\eta$}

We already know by Proposition~\ref{prop:admissible} that all finite blocks of 0-s and 1-s appearing in the sequence $\eta$ have to be $\B$-admissible. The following theorem shows that, conversely, any $\B$-admissible finite block does appear in $\eta$, and this with a positive asymptotic frequency given by the probability $\nub$ of the corresponding cylinder. 
\begin{theo}
\label{thm:eta_generic}
The sequence $\eta$ is generic for the probability $\nu_\B$, \textit{i.e.}
  for any cylinder set $C\subset X$, we have 
  \begin{equation}
    \label{eq:eta_nus-generic}
    \dfrac{1}{N} \sum_{0\le n<N} \ind{C}(S^n\eta) \tend{N}{\infty} \nub(C).
  \end{equation}
\end{theo}

\begin{proof}
 Recalling that $\eta=\varphi(\underline{\omega})$, we first observe that \eqref{eq:eta_nus-generic} is equivalent to
  \begin{equation}
    \label{eq:eta_nus-generic-bis}
    \dfrac{1}{N} \sum_{0\le n<N} \ind{\varphi^{-1}(C)}(T^n\underline{\omega}) \tend{N}{\infty} \PP\bigl(\varphi^{-1}(C)\bigr).
  \end{equation}
It is well known that, under the action of a minimal rotation on a compact group, each element of the group is generic for the normalized Haar measure on the group. So, if $\varphi$ was continuous, \eqref{eq:eta_nus-generic-bis} would be immediate. Unfortunately $\varphi$ is not continuous, and we have to work a little harder to prove the claim.

We note that it is enough to prove~\eqref{eq:eta_nus-generic-bis} in the case where the cylinder set $C$ is of the form $C_B^0$ for some finite subset $B\subset\NNS$. (Recall that $C_B^0$ is the set of sequences $(x_n)\in X$ satisfying $x_n=0$ for each $n\in B$.)
Indeed, if~\eqref{eq:eta_nus-generic-bis} is valid for each such $C_B^0$, then applying Lemma~\ref{lemma:general} with the family of events $E_n\egdef\{\omega:\varphi(\omega)_n=0\}$, both for $\PP$ and for the empirical measure $\frac{1}{N} \sum_{0\le n<N} \delta_{T^n\underline{\omega}}$, we can prove that \eqref{eq:eta_nus-generic-bis} holds for any cylinder set $C_{A,B}$.

Let $B$ be a finite subset of $\NNS$. We 
observe that \eqref{eq:eta_nus-generic-bis} for the cylinder set $C_B^0$ means
\begin{equation}
    \label{eq:eta_nus-generic-ter}
    \dfrac{1}{N} \sum_{0\le n<N} \ind{\varphi^{-1}(C_B^0)}(T^n\underline{\omega}) \tend{N}{\infty} \PP\bigl(\varphi^{-1}(C_B^0)\bigr).
  \end{equation}
We can write
$\varphi^{-1}(C_B^0)$ as the union of the increasing sequence of events $H_\ell$, where
\[
    H_\ell \egdef \bigl\{\omega\in\Omega: \forall m\in B,\ \exists 1\le k_m \le \ell,\  \omega_{k_m}+m=0\ [b_{k_m}]\,\bigr\}.
\]
We have $\PP(H_\ell)\tend{\ell}{\infty}\PP\bigl(\varphi^{-1}(C_B^0)\bigr)$.
Since $\ind{H_\ell}(\omega)$ only depends on the coordinates $\omega_k$, $1\le k\le \ell$,  and the action of $T$ on the finite group $\prod_{1\le k\le \ell}\ZZ/b_k\ZZ$ being, by ergodicity of $T$, a cyclic permutation of the elements of this group, we have 
\begin{equation}
\label{eq:ergodic_thm_in_finite_group}
  \dfrac{1}{N} \sum_{0\le n<N} \ind{H_\ell}(T^n\underline{\omega}) \tend{N}{\infty} \PP(H_\ell).
\end{equation}
To prove \eqref{eq:eta_nus-generic-ter}, it only remains to show that
\begin{equation}
    \label{eq:error_term}
    \sup_{N\ge 1} \dfrac{1}{N} \sum_{0\le n<N} \ind{\varphi^{-1}(C_B^0)\setminus H_\ell}(T^n\underline{\omega}) \tend{\ell}{\infty} 0.
\end{equation}
If $\omega\in \varphi^{-1}(C_B^0)\setminus H_\ell$, then there exist $m\in B$ and $k>\ell$ such that $\omega_k + m=0\ [b_k]$. Therefore we can write
\begin{equation}
\label{eq:majoration_error_term}
  \ind{\varphi^{-1}(C_B^0)\setminus H_\ell} \le \sum_{m\in B}\sum_{k>\ell} f_{m,k},
\end{equation}
where \[
        f_{m,k}(\omega)\egdef\begin{cases}
                               1 \text{ if }\omega_k+m=0\ [b_k],\\
                               0 \text{ otherwise.}
                             \end{cases}
      \]
For a fixed $m\in B$ and $k>\ell$, let us consider the sequence $\bigl(f_{m,k}(T^n\underline{\omega})\bigr)_{n\in\NNS}$. This is a sequence of 0-s and 1-s, in which two occurences of 1 are separated by a multiple of $b_k$. Moreover, as $\underline{\omega}_k=0$, 
the sequence shows  no 1 in the interval $1\le n< b_k-\sup B$. It follows that, provided $k$ is large enough so that $b_k>2\sup B$ (which is guaranted if  $\ell$ is large enough),  for any $N\ge 1$ we have the inequality
\begin{equation}
\label{eq:2_sur_bk}
  \dfrac{1}{N} \sum_{0\le n<N} f_{m,k}(T^n\underline{\omega}) \le \dfrac{2}{b_k}.
\end{equation}
We then have, for $\ell$ large enough,
\[
  \dfrac{1}{N} \sum_{0\le n<N} \ind{\varphi^{-1}(C_B^0)\setminus H_\ell}(T^n\underline{\omega})
  \le 2|B| \sum_{k>\ell}\dfrac{1}{b_k},
\]
which proves~\eqref{eq:error_term}.
\end{proof}

\begin{corollary}
\label{cor:subshift}
  The subshift $\overline{\{S^n\eta: n\in\NNS\}}$ generated by $\eta$ coincides with the subshift $X_\B$ of $\B$-admissible sequences.
\end{corollary}

Applying~\eqref{eq:eta_nus-generic} with the cylinder sets $C_A^1$ when $A=\{x\in X: x_1=x_3=1\}$ or $A=\{x\in X: x_1=x_2=1\}$ gives also the following interesting consequence of Theorem~\ref{thm:eta_generic}. 

\begin{corollary}[Twin $\B$-free integers]
 \label{cor:twin}
There always exist infinitely many integers $n$ such that $n$ and $n+2$ are simultaneously $\B$-free. 
If $2\notin\B$, there exist infinitely many integers $n$ such that $n$ and $n+1$ are simultaneously $\B$-free. 
\end{corollary}

This result was already known in the context of $r$-free integers by Mirsky's theorem~\cite{mirsky}, and the question of the speed of convergence of the density has been much studied using the circle method (see e.g.~\cite{heath-brown, dietman-marmon} and references therein). But to our knowledge its extension to the context of $\B$-free numbers is new. It would be natural to ask whether our method can be adapted to provide a speed of convergence.

\subsection{Averaging along arithmetic subsequences}
The method used to prove Theorem~\ref{thm:eta_generic} can also be useful if we are interested in the frequency of blocks in $\eta$ along arithmetic subsequences. Of course, all admissible blocks do not appear along any arithmetic subsequence since, for instance, $\eta_{nb_1}=0$ for each $n\ge0$. But, as an example of what we can say in this setting, we prove the following result.

\begin{theo}
\label{thm:arithmetic-subsequences}
 Let $p$ be a prime number. Then, there exists an integer $m\ge1$, depending only on $p$ and $\B$, such that for any integer $s\ge1$, we have for each cylinder set $C\subset X$
\begin{equation}
 \label{eq:arithmetic}
\dfrac{1}{p^m} \sum_{r=0}^{p^m-1}\ \dfrac{1}{N} \sum_{0\le n<N} \ind{C}(S^{np^s+r}\,\eta) \tend{N}{\infty} \nub(C).
\end{equation}
\end{theo}

\begin{proof}
 We follow the same argument as in the proof of Theorem~\ref{thm:eta_generic}, with $T^{p^s}$ in place of $T$. For each $b_k$ which is not divisible by $p$, the addition of $p^s$ in $\ZZ/b_k\ZZ$ is isomorphic to the addition of 1. The only modification in the argument happens if there exists some $b_{k_0}$ which is divisible by $p$ (in this case, $b_{k_0}$ is unique by~\eqref{eq:relatively prime}). 
 Then, let $m$ be the highest power of $p$ dividing $b_{k_0}$, define $m'\egdef\min\{m,s\}$ and write $b_{k_0}=q p^{m'}$. The Euclidean division by $p^{m'}$ of an element $j\in\ZZ/b_k\ZZ$ (identified with $\{0,\ldots,b_k-1\}$) provides a natural isomorphism between $\ZZ/b_k\ZZ$ and the direct product $\ZZ/q\ZZ\times\ZZ/p^{m'}\ZZ$. Besides, the addition of $p^s$ in $\ZZ/b_{k_0}\ZZ$ is isomorphic to the addition of $(1,0)$ in $\ZZ/q\ZZ\times\ZZ/p^{m'}\ZZ$.
 Now everything works as in the preceding proof, except that we have added a supplementary cyclic group $\ZZ/p^{m'}\ZZ$ on which $T^{p^s}$ acts as the identity. Therefore, assuming $\ell\ge k_0$, \eqref{eq:ergodic_thm_in_finite_group} becomes
\begin{equation*}
\label{eq:ergodic_thm_in_finite_group_bis}
  \dfrac{1}{p^{m'}} \sum_{r=0}^{p^{m'}-1} \dfrac{1}{N} \sum_{0\le n<N} \ind{H_\ell}(T^{np^s+r}\underline{\omega}) \tend{N}{\infty} \PP(H_\ell).
\end{equation*}
Since $m'\le m$, $p^{m}$ is a multiple of $p^{m'}$ and we also have
\[
  \dfrac{1}{p^{m}} \sum_{r=0}^{p^{m}-1} \dfrac{1}{N} \sum_{0\le n<N} \ind{H_\ell}(T^{np^s+r}\underline{\omega}) \tend{N}{\infty} \PP(H_\ell).
\]
The end of the proof goes as in the proof of Theorem~\ref{thm:eta_generic}.
\end{proof}

\subsection{Distribution of $\B$-free numbers in short intervals}
We want to show here that the proof of Theorem~\ref{thm:eta_generic} can also be adapted to the case where averaging is done along intervals of the form $[N,N+\sqrt{N})$. 
The following theorem will show that the frequency of blocks in $\eta$ in these ``short intervals'' also conforms to the measure $\nub$,  provided we assume a further hypothesis on the set $\B$: We suppose here that
\begin{equation}
 \label{eq:new_hypothesis}
  \exists M,\ \forall x\text{ large enough, } \bigl| \B\cap [x,x+\sqrt{x}] \bigr| \le M.
\end{equation}

Note that the above assumption holds if, for example, there exists an integer $r\ge2$ such that $\B=\{p^r: p \text{ prime}\}$.  Indeed, the length of the interval $\left[x^{1/r},(x+\sqrt{x})^{1/r}\right]$ goes to 0 as $x\to\infty$, hence this interval contains at most one prime number for $x$ large.

\begin{theo}
 \label{thm:short_intervals}
 Assume that~\eqref{eq:new_hypothesis} holds. Then, for any cylinder set $C\subset X$, we have 
  \begin{equation}
    \label{eq:generic_on_short_intervals}
    \dfrac{1}{\sqrt{N}} \sum_{N\le n<N+\sqrt{N}} \ind{C}(S^n\eta) \tend{N}{\infty} \nub(C).
  \end{equation}
\end{theo}

\begin{proof}
 We will follow the same lines as in the proof of Theorem~\ref{thm:eta_generic}. Again, it is enough to prove that, if $C$ is a cylinder set of the form $C_B^0$, then the following convergence holds:
  \begin{equation}
    \label{eq:generic_on_short_intervals_bis}
    \dfrac{1}{\sqrt{N}} \sum_{N\le n<N+\sqrt{N}}\ind{\varphi^{-1}(C_B^0)}(T^n\underline{\omega}) \tend{N}{\infty} \PP\bigl(\varphi^{-1}(C_B^0)\bigr).
  \end{equation}
For $\ell\ge1$, let $H_\ell$ be defined as in the proof of Theorem~\ref{thm:eta_generic}. For the same reason as in~\eqref{eq:ergodic_thm_in_finite_group}, we also have
\[
  \dfrac{1}{\sqrt{N}} \sum_{N\le n<N+\sqrt{N}} \ind{H_\ell}(T^n\underline{\omega}) \tend{N}{\infty} \PP(H_\ell),
\]
and it only remains to show that
\begin{equation}
    \label{eq:error_term_short_intervals}
    \limsup_{N\to\infty} \dfrac{1}{\sqrt{N}} \sum_{N\le n<N+\sqrt{N}} \ind{\varphi^{-1}(C_B^0)\setminus H_\ell}(T^n\underline{\omega}) \tend{\ell}{\infty} 0.
\end{equation}
We can use again~\eqref{eq:majoration_error_term}. Observing that 
\[  
f_{m,k}(T^n\underline{\omega})=\begin{cases}
                                 1\ \text{ if $b_k$ divides $n+m$,}\\
                                 0\ \text{ otherwise,}
                               \end{cases}
\]
we get that for any fixed $N\ge1$,
\begin{multline}
    \label{eq:error_term_short_intervals_bis}
  \dfrac{1}{\sqrt{N}} \sum_{N\le n<N+\sqrt{N}} \ind{\varphi^{-1}(C_B^0)\setminus H_\ell}(T^n\underline{\omega})  \\
  \le \sum_{m\in B}  \sum_{k>\ell}\dfrac{1}{\sqrt{N}}  \sum_{N\le n<N+\sqrt{N}} \ind{n+m=0\ [b_k]},
\end{multline}
where $\ind{n+m=0\ [b_k]}$ has to be seen as a function of $n$ which is either equal to $1$ if $n+m=0\ [b_k]$, or to 0 otherwise.

Let $L\ge2$ be a large enough number, to be precised later. For a fixed $N$, we distinguish two kinds of integers $b_k$:
\begin{itemize}
  \item We say that $b_k$ is \emph{small} if $b_k\le L\sqrt{N}$.
  \item We say that $b_k$ is \emph{large} if $b_k> L\sqrt{N}$.
\end{itemize}

\smallskip

Let us consider first the contribution of large $b_k$-s. For such a $b_k$, for any $m\in B$,  there exists at most one $n\in[N,N+\sqrt{N})$ such that $n+m=0\ [b_k]$. It follows that
\[
  \sum_{N\le n<N+\sqrt{N}} \ind{n+m=0\ [b_k]} \in\{0,1\}, 
\]
and the above expression takes value 1 if and only if there exists $N\le n<N+\sqrt{N}$ such that $n+m$ is a multiple of $b_k$. We rewrite the latter condition as follows: There exists an integer $j\ge1$ with $jb_k\in[N+m,N+m+\sqrt{N})$.
But, since $b_k$ is large, this implies 
\[
  L\sqrt{N} < b_k < \dfrac{N+m+\sqrt{N}}{j},
\]
hence, for $N$ large enough,
\[
  j < \dfrac{1}{L} \left(  \sqrt{N}+\dfrac{m}{\sqrt{N}}+1\right) < \dfrac{2\sqrt{N}}{L}.
\]
Since $\sqrt{N}/j\le \sqrt{(N+m)/j}$, it follows by assumption~\eqref{eq:new_hypothesis} that
\begin{multline}
  \sum_{\stack{k>\ell}{b_k\text{ large}}}\dfrac{1}{\sqrt{N}}  \sum_{N\le n<N+\sqrt{N}} \ind{n+m=0\ [b_k]} \\
   \le \dfrac{1}{\sqrt{N}}\sum_{1\le j < 2\sqrt{N}/L} \left| \B\cap \left[ (N+m)/j, (N+m+\sqrt{N})/j \right)\right| 
   \le \dfrac{2M}{L},
\end{multline}
hence we can assume that the contribution of large $b_k$-s in the right-hand side of \eqref{eq:error_term_short_intervals_bis} is arbitrarily small by choosing $L$ large enough.

\smallskip

Then consider the case of a small $b_k$. If $b_k>\sqrt{N}$, again for each $m\in B$ there exists at most one $n\in[N,N+\sqrt{N})$ such that $n+m=0\ [b_k]$, and we have 
\[
  \dfrac{1}{\sqrt{N}}  \sum_{N\le n<N+\sqrt{N}} \ind{n+m=0\ [b_k]} \le \dfrac{1}{\sqrt{N}} \le \dfrac{L}{b_k}.
\]
On the other hand, if $\sqrt{N}\ge b_k$,  since different integers $n$ satisfying $n+m=0\ [b_k]$ are separated by multiples of $b_k$, we have as in~\eqref{eq:2_sur_bk}
\[
  \dfrac{1}{\sqrt{N}}  \sum_{N\le n<N+\sqrt{N}} \ind{n+m=0\ [b_k]} \le \dfrac{2}{b_k} \le \dfrac{L}{b_k}.
\]
We finally get, for any $m\in B$,
\[
  \sum_{\stack{k>\ell}{b_k\text{ small}}}\dfrac{1}{\sqrt{N}}  \sum_{N\le n<N+\sqrt{N}} \ind{n+m=0\ [b_k]} \le L\sum_{k>\ell} \dfrac{1}{b_k},
\]
which can be made as small as desired by choosing $\ell$ large enough. This demonstrates the validity of \eqref{eq:error_term_short_intervals}, and concludes the proof of the theorem.
\end{proof}

One can observe that, as in Corollary~\ref{cor:twin}, the previous theorem proves the abundance of twin $\B$-free numbers in short intervals, provided that condition~\eqref{eq:new_hypothesis} is satisfied.

\section{Entropy of the subshift $X_\B$}

The purpose of this section is to compute the topological entropy of the subshift $X_\B$ generated by $\eta$, which measures the exponential growth of the number of subwords of length $n$ in $\eta$. We denote by $\gamma(n)$ this number:
\[
  \gamma(n) \egdef \bigl| \left\{  W\in\{0,1\}^n: W\text{ is $\B$-admissible}  \right\} \bigr|.
\]

For each $K\ge1$, let $\B_K\egdef\{b_1,\ldots,b_K\}\subset \B$. We define $\B_K$-admissibility in the same way as $\B$-admissibility, restricting condition~\eqref{eq:def-admissible} to $k\in\{1,\ldots,K\}$. We then set
\[
  \gamma_K(n) \egdef \bigl| \left\{  W\in\{0,1\}^n: W\text{ is $\B_K$-admissible}  \right\} \bigr|.
\]
Since $\B$-admissibility obviously implies $\B_K$-admissibility for all $K\ge1$, we have
\begin{equation}
  \label{eq:gammas} 
  \forall K\ge1,\ \gamma(n)\le \gamma_K(n).
\end{equation}
For a fixed $n$, the set of $\B_K$-admissible words of length $n$ decreases, as $K\to\infty$, to the set of $\B$-admissible words of length $n$, so that there exists $K(n)$ satisfying
\begin{equation}
  \label{eq:K(n)}
  \gamma(n)=\gamma_{K(n)}(n).
\end{equation}

We want now to estimate the quantity $\gamma_K(n)$ for $K\ge1$. For $z_k\in\ZZ/b_k\ZZ$, $1\le k\le K$, set
\[
  Z(z_1,\ldots,z_K) \egdef \{n\in\NNS: \exists 1\le k\le K,\ n=z_k\ [b_k]\}
\]
Notice that the set $\{1,...,n\}\setminus Z(z_1,...,z_K)$
is $\B_K$-admissible.

\begin{lemma}
  \label{lemma:chinese}
  If $n$ is a multiple of $b_1\cdots b_K$,
  \[
    \dfrac{1}{n} \bigl|  \{1,\ldots,n\}\setminus Z(z_1,\ldots,z_K) \bigr| = \prod_{k=1}^K \left( 1 - \dfrac{1}{b_k}\right).
  \]
\end{lemma}
\begin{proof}
If $A$ is a set of $b_1\cdots b_K$ consecutive integers, the Chinese Remainder Theorem tells us that the map $\theta:\ A\to\prod_{k=1}^K\ZZ/b_k\ZZ$ defined by
\[ \theta(j)\egdef \bigl( j \ [b_1],\ldots,j\, [b_K]\bigr)
\]
is bijective. But $\theta\bigl(A\setminus Z(z_1,\ldots,z_K)\bigr)$ is the set of $w=(w_1,\ldots,w_K)\in \prod_{k=1}^K\ZZ/b_k\ZZ$ satisfying
\[
 \forall 1\le k\le K,\ w_k\neq z_k.
\]
Since the cardinality of the latter set is $\prod_{k=1}^K(b_k-1)$, we get
\[
  \dfrac{|A\setminus Z(z_1,\ldots,z_K)|}{|A|} = \prod_{k=1}^K \left( 1 - \dfrac{1}{b_k}\right),
\]
and the lemma follows.
\end{proof}

\begin{prop}
  \label{prop:estimation gamma_K}
Assume that $n$ is a multiple of $b_1\cdots b_K$. Then
\begin{equation}
  \label{eq:estimation gamma_K}
  2^{n\prod_{k=1}^K \left( 1 - \frac{1}{b_k}\right)} 
	    \le \gamma_K(n)
		    \le  2^{n\prod_{k=1}^K \left( 1 - \frac{1}{b_k}\right)} \prod_{k=1}^K b_k.
\end{equation}
\end{prop}
\begin{proof}
Here is a way to produce a $\B_K$-admissible word $W=w_1\ldots w_n$ of length~$n$:
\begin{itemize}
  \item[Step 1.] Choose $(z_1,\ldots,z_K)\in \prod_{k=1}^K \ZZ/b_k\ZZ$, and set $w_j=0$ for $j\in Z(z_1,\ldots,z_K)$;
  \item[Step 2.] Complete the word by choosing arbitrarily $w_j\in\{0,1\}$ for each $j\in  \{1,\ldots,n\}\setminus Z(z_1,\ldots,z_K)$.
\end{itemize}
Observe that, if we have fixed $(z_1,\ldots,z_K)\in \prod_{k=1}^K \ZZ/b_k\ZZ$, 
then by Lemma~\ref{lemma:chinese}, the freedom given in Step~2 produces 
\[
  2^{n\prod_{k=1}^K \left( 1 - \frac{1}{b_k}\right)}
\]
different $\B_K$-admissible words. On the other hand, any $\B_K$-admissible word can be obtained in this way, and since there are $\prod_{k=1}^K b_k$ different possible choices in Step~1, we get the other inequality.
\end{proof}

\begin{theo}
  \label{thm:entropy}
  The topological entropy of the subshift $X_\B$ is given by
  \[
    h_{\text{top}} (X_\B) = \prod_{k=1}^\infty \left( 1 - \frac{1}{b_k}\right)
  \]
\end{theo}

\begin{proof}
  We have to prove that 
  \begin{equation}
    \label{eq:entropy}
    \lim_{n\to\infty} \dfrac{1}{n}\log_2\gamma(n) = \prod_{k=1}^\infty \left( 1 - \frac{1}{b_k}\right).
  \end{equation}
Let $K\ge 1$. For each $n$ which is a multiple of the product $b_1\cdots b_K$, we have
\begin{align*}
  \dfrac{1}{n}\log_2\gamma(n) & \le \dfrac{1}{n}\log_2\gamma_K(n) \quad\text{(by \ref{eq:gammas})}\\
  & \le \prod_{k=1}^K \left( 1 - \frac{1}{b_k}\right) + \dfrac{1}{n}\log_2\prod_{k=1}^K b_k\quad\text{(by Proposition \ref{prop:estimation gamma_K}).}\\
\end{align*}
By first choosing $K$ large enough, and then letting $n$ go to $\infty$, we get
\begin{equation}
  \label{eq:majoration_entropy}
  \limsup_{n\to\infty}\dfrac{1}{n}\log_2\gamma(n) \le \prod_{k=1}^\infty \left( 1 - \frac{1}{b_k}\right)
\end{equation}

On the other hand, using the subadditivity of $\log_2\gamma_K$:
\[
  \dfrac{1}{n}\log_2\gamma_K(n) \ge \dfrac{1}{nm}\log_2\gamma_K(nm),
\]
and recalling from~\eqref{eq:K(n)} that $\gamma(n)=\gamma_{K(n)}(n)$, we get, for all $n\ge 1$,
\begin{align*}
  \dfrac{1}{n}\log_2\gamma(n) & = \dfrac{1}{n}\log_2\gamma_{K(n)}(n) \\
  & \ge \dfrac{1}{nb_1\cdots b_{K(n)}} \log_2\gamma_{K(n)}(nb_1\cdots b_{K(n)}) \\
  & \ge \prod_{k=1}^{K(n)} \left( 1 - \frac{1}{b_k}\right) \quad\text{(by Proposition \ref{prop:estimation gamma_K})}\\
  & \ge \prod_{k=1}^{\infty} \left( 1 - \frac{1}{b_k}\right).
\end{align*}
Together with~\eqref{eq:majoration_entropy}, this proves~\eqref{eq:entropy}.
\end{proof}

The computation of the topological entropy of $X_\B$ has been done by Peckner in \cite{peckner} in the context of the square-free flow, that is when $\B$ is the set of squares of primes. In this paper, Peckner also proves that there exists a unique invariant probability measure on $X_\B$ with maximal entropy, and describes the structure of the measurable dynamical system defined by this measure. 
\begin{question}
  Can we generalize Peckner's results in the $\B$-free numbers setting?
\end{question}

\section{A generalized version of the Möbius function}

In the case where $\B$ is the set of squares of prime numbers, we have $\eta_n=\mob(n)^2$, where $\mob:\ n\mapsto \mob(n)$ is the classical Möbius function:
\[
  \mob(n) \egdef \begin{cases}
              0 \text{ if $n$ is divisible by a square,}\\
              (-1)^{d_n} \text{ otherwise, where $d_n$ is the number of prime factors of $n$.}
            \end{cases}
\]
We would like to interpret also this Möbius function (or its counterpart in a more general context) in a similar dynamical setting.

\smallskip

In this purpose, we assume now that for each $k\ge1$, $b_k=a_k^2$, where the integers $a_k$, $k\ge1$,  are pairwise relatively prime (one can have in mind that $(a_k)$ is an increasing sequence of prime numbers for example). 

We are interested in the sequence $\mu=(\mu_n)_{n\in\NNS}$ defined as follows: First, set, for each $n\in\NNS$,
\[ 
  \delta_n\egdef |\{k\ge1: a_k\text{ divides }n\}|.
\]
Next, define
\[
  \pi_n\egdef(-1)^{\delta_n}
\]
and finally
\[
  \mu_n\egdef \eta_n \pi_n.
\]
Of course,  if $(a_k)$ is the whole sequence of prime numbers, then $\mu$ coincides with the Möbius function. We shall refer to this situation as the \emph{Möbius case}. In this case, observe that $\pi$ is a multiplicative arithmetic function, which presents some analogy with the classical Liouville function (but the latter counts the number of prime divisors with multiplicity, which is not the case of our sequence $\pi$).

\smallskip

We now need to work in the spaces $Y\egdef\{-1,1\}^{\NNS}$, and $Z\egdef \{-1,0,1\}^{\NNS}$. If the context requires to clarify on which space, $X$, $Y$, or $Z$, acts the shift map, we shall use respectively the notations  $S_X$,  $S_Y$ and $S_Z$. 

\subsection{An interpretation of the Chowla conjecture}

We say that the sequence $(\mu_n)$ defined above satisfies the \emph{Chowla property} if, for any $r\ge1$, any  natural numbers  $0\leq s_1<\ldots<s_r$ and $i_1,\ldots,i_r\in\{1,2\}$ not all equal to~2, 
\begin{equation}
  \label{eq:Chowla property}
  \dfrac{1}{N} \sum_{1\le n\le N} \mu_{n+s_1}^{i_1}\cdots \mu_{n+s_r}^{i_r} \tend{N}{\infty} 0.
\end{equation}
The well-known Chowla conjecture asserts that, in the Möbius case, $(\mu_n)=(\mob(n))$ satisfies the above property.

\smallskip

Now we want to highlight the links between the Chowla property and the unpredictability of the sequence $\pi\egdef(\pi_n)_{n\in\NNS}$. Let us consider the most unpredictable shift-invariant probability measure on  $Y$, namely the Bernoulli measure denoted by $\beta$ under which the coordinates are independent and uniformly distributed on $\{-1,1\}$. Assume that $\pi$ is generic for $\beta$, \textit{i.e.} that for each cylinder set $C=\{y\in Y: y_1=\alpha_1,\ldots, y_m=\alpha_m\}$,
\begin{equation}
  \label{eq:pi is beta generic}
  \dfrac{1}{N} \sum_{0\le n<N} \ind{C}(S^n \pi) \tend{N}{\infty} \beta(C)=\dfrac{1}{2^{m}}.
\end{equation}

Recall from Theorem~\ref{thm:eta_generic} that the sequence $\eta=\mu^2$ is generic for the shift-invariant measure $\nu_\B$. Then it follows that the pair $(\eta,\pi)$ is generic for the product measure $\nu_\B\otimes\beta$ on $X\times Y$. Indeed, any weak limit $\rho$ of a subsequence
\[
  \left( \dfrac{1}{N_j}\sum_{0\le n<N_j} \delta_{(S_X^n \eta, S_Y^n \pi)} \right)_{j\ge 1}
\]
has to be a joining of the two dynamical systems $(X,\nu_\B,S_X)$ and $(Y,\beta,S_Y)$. But the former is a discrete-spectrum dynamical system, and the latter is a Bernoulli shift, hence these two systems are disjoint and $\rho$ can only be the product measure $\nu_\B\otimes\beta$ (see \cite{furstenberg}).
As a consequence, since $\mu_n=\eta_n\pi_n$,  we get that $\mu$ is generic for the probability measure $\nu_M$ defined on $Z$ as the pushforward measure of $\nu_\B\otimes\beta$ by the continuous map $\Theta: (x,y)\in X\times Y\mapsto x\cdot y\in Z$ where $(x\cdot y)_n\egdef x_n y_n$. We can easily compute the measure $\nu_M$ of each cylinder set $C=\{z\in Z: z_1=\alpha_1,\ldots,z_m=\alpha_m\}\subset Z$. First, let us denote by $C^2$ the associated cylinder set in $X$:
\[
  C^2\egdef\{x\in X: x_1=\alpha_1^2,\ldots,x_m=\alpha_m^2\},
\]
and set $\lambda(C)\egdef \sum_{1\le n\le m}\alpha_n^2$. Then we have
\[
  \nu_M(C) = \dfrac{1}{2^{\lambda(C)}}\, \nu_\B(C^2).
\]

\begin{theo}
  \label{thm:chowla and genericity}
  We have $(i)\Longrightarrow (ii) \Longleftrightarrow (iii)$, where \\
  $(i)$ The sequence $\pi$ is generic for $\beta$.\\
  $(ii)$ The sequence $\mu$ is generic for $\nu_M$.\\
  $(iii)$ The sequence $\mu$ satisfies the Chowla property.
\end{theo}

(Note that the equivalence between $(ii)$ and $(iii)$ has been stated in \cite{sarnak} in the context of the classical Möbius function.)
\begin{proof}
  We have already seen just before the theorem that $(i)\Longrightarrow (ii)$, and it only remains to prove the equivalence of $(ii)$ and $(iii)$. 
  
Consider $f: z\in Z\mapsto f(z)\egdef z_1$,  and notice that
\[
  \dfrac{1}{N}\sum_{1\le n\le N} \mu_{n+s_1}^{i_1}
\cdots \mu_{n+s_r}^{i_r} 
= \dfrac{1}{N}\sum_{0\le n<N}\!\big( f^{i_1}\!\circ\! S^{s_1}\cdots  f^{i_r}\circ\! S^{s_r}\!\big)(\!S^n\mu\!).
\]
If we assume that $(ii)$ holds, then the above expression converges to the integral 
\begin{align*}
 &\int_{Z} f^{i_1}\!\circ\! S^{s_1}\cdots  f^{i_r}\circ\! S^{s_r}\,d\nu_M\\
 &=\sum_{\alpha_1,\ldots,\alpha_r=\pm1} \alpha_1^{i_1}\cdots \alpha_r^{i_r} \ \nu_M\Bigl(\bigl\{z\in Z:(z_{s_1},\ldots,z_{s_r})=(\alpha_1,\ldots,\alpha_r)\bigr\}\Bigr)\\
 &= \left(\sum_{\alpha_1,\ldots,\alpha_r=\pm1} \alpha_1^{i_1}\cdots \alpha_r^{i_r} \right)\frac{1}{2^{r}}\ 
\nu_\B\Bigl(\bigl\{x\in X: x_{s_1}=\ldots=x_{s_r}=1\bigr\}\Bigr).
\end{align*}
But the first factor in the right-hand side is equal to 0 as soon as at least one of the exponents $i_j$ is equal to~1, hence we get $(iii)$.

To obtain the other direction, note that $(iii)$ tells us that we have a correct limit for the products $f^{i_1}\circ S^{s_1}\cdot\ldots\cdot f^{i_r}\circ S^{s_r}$ when the exponents are not all even. But in the case of even exponents, the limit is always correct by Theorem~\ref{thm:eta_generic}. Since the family of functions $f^{i_1}\circ S^{s_1}\cdot\ldots\cdot f^{i_r}\circ S^{s_r}$ (when $r$, the $s_j$ and the exponents $i_j$ vary) separates points and is closed under taking products, (ii) follows by using the Stone-Weierstrass theorem.
\end{proof}

\subsection{Genericity of the sequence $\pi$}

Now, we consider the compact abelian group $\Omega'\egdef \prod_{k\ge 1}\ZZ/a_k\ZZ$, we denote by $\PP'$ the Haar measure on $\Omega'$, and by $T'$ the addition of $(1,1,\ldots)$ in $\Omega'$. Observe that the sequence $\pi=(\pi_n)_{n\in\NNS}$ can also be observed in the dynamical system $(\Omega',T',\PP')$: Indeed, consider the map $\Delta:\ \Omega'\to\NNS\cup\{\infty\}$ defined by 
\[
  \Delta(\omega') \egdef \bigl|\{k\ge1: \omega'_k = 0 \}\bigr|.
\]
Then, denoting by $\underline{\omega'}$ the point $(0,0,\ldots)\in\Omega'$, we have for $n\in\NNS$
\[
  \delta_n = \Delta(T'^n \underline{\omega'}),\text{ and}\quad \pi_n = (-1)^{\Delta(T'^n \underline{\omega'})}.
\]

The main difficulty lies in the fact that the quantity $\Delta(\omega')$ can take infinite values: For example, $\Delta(\underline{\omega'})=\infty$ (while $\Delta(T'^n \underline{\omega'})$ is finite for $n\ge1$). In particular, if $\sum_{k\ge1} 1/a_k=\infty$ (which happens in the Möbius case), by independence of the events $(\omega'_k=0)$, $k\ge 1$, the second Borel-Cantelli lemma gives that $\Delta(\omega')=\infty$ for $\PP'$-almost all $\omega'$.

\smallskip

However, we can place ourselves in the case where 
\begin{equation}
  \label{eq:Sigma_is_finite}
  \Sigma\egdef\sum_{k\ge1}\dfrac{1}{a_k} < \infty.
\end{equation}
Note that, in this situation, what we proved in the context of $\Omega$ is also valid in the context of $\Omega'$.

Under the hypothesis~\eqref{eq:Sigma_is_finite}, the Borel-Cantelli lemma ensures that $\Delta(\omega')<\infty$ $\PP'$-a.s., and we can consider the map $\psi:\ \Omega'\to Y$ defined $\PP'$-almost surely by
\[
  \psi(\omega') \egdef \Bigl( (-1)^{\Delta(T'\omega')}, (-1)^{\Delta(T'^2\omega')},\ldots, (-1)^{\Delta(T'^n\omega')},\ldots\Bigr).
\]

Let us denote by $\nup$ the pushforward measure of $\PP'$ by $\psi$. Then $\nup$ is a shift-invariant probability measure on $Y$.

\begin{theo}
\label{thm:pi_generic}
  Assume that \eqref{eq:Sigma_is_finite} holds. Then the sequence $\pi$ is generic for $\nup$, \textit{i.e.} for any cylinder set $C$, we have
  \begin{equation}
    \label{eq:pi_num-generic}
    \dfrac{1}{N} \sum_{0\le n<N} \ind{C}(S^n\pi) \tend{N}{\infty} \nup(C).
  \end{equation}
\end{theo}

\begin{proof}
  The arguments are similar to those in the proof of Theorem~\ref{thm:eta_generic}. Let $B$ be a finite subset of $\NNS$, let $(\alpha_m)_{m\in B}\in\{-1,1\}^B$, and let $C\subset \{-1,1\}^\NNS$ be the cylinder set 
  \[
    C\egdef\left\{y=(y_n)_{n\in\NNS}\in\{-1,1\}^\NNS: \forall m\in B,\ y_m=\alpha_m\right\}.
  \]
We have to prove the convergence
\[
  \dfrac{1}{N} \sum_{0\le n<N} \ind{\psi^{-1}(C)}(T'^n\underline{\omega'}) \tend{N}{\infty} \PP'\bigl(\psi^{-1}(C)\bigr).
\]
Set, for $\omega'\in\Omega'$ and $\ell\ge 1$
\[
  \Delta_\ell(\omega') \egdef  \big|\{k: 1\le k\le \ell,\ \omega'_k = 0\}\big|.
\]
Under the assumption~\eqref{eq:Sigma_is_finite}, $\Delta_\ell(\omega')\tend{\ell}{\infty}\Delta(\omega)$ $\PP'$-a.s.
Hence, considering the event
\[
  G_\ell \egdef \left\{\omega'\in\Omega': \forall m\in B,\ (-1)^{\Delta_\ell(T'^m\omega')}=\alpha_m\right\},
\]
we have 
\[
  \ind{G_\ell}\tend[\PP'\text{-a.s.}]{\ell}{\infty} \ind{\psi^{-1}(C)},
\]
and by the dominated convergence theorem
\[
  \PP'(G_\ell)\tend{\ell}{\infty}\PP'\bigl(\psi^{-1}(C)\bigr).
\]
For the same reason as in~\eqref{eq:ergodic_thm_in_finite_group}, we have
\[
  \dfrac{1}{N} \sum_{0\le n<N} \ind{G_\ell}(T'^n\underline{\omega'}) \tend{N}{\infty} \PP'(G_\ell),
\]
and it only remains to show that
\begin{equation}
    \label{eq:error_term_lambda}
    \sup_{N\ge 1} \dfrac{1}{N} \sum_{0\le n<N} \ind{\psi^{-1}(C)\vartriangle G_\ell}(T'^n\underline{\omega'}) \tend{\ell}{\infty} 0.
\end{equation}
But, as in the proof of Theorem~\ref{thm:eta_generic}, if $\omega'\in \psi^{-1}(C)\vartriangle G_\ell$, then there exist $m\in B$ and $k>\ell$ such that $\omega'_k + m=0\ [a_k]$. We get, as in~\eqref{eq:majoration_error_term},
\begin{equation*}
  \ind{\psi^{-1}(C)\vartriangle G_\ell} \le \sum_{m\in B}\sum_{k>\ell} f'_{m,k},
\end{equation*}
where \[
        f'_{m,k}(\omega')\egdef\begin{cases}
                               1 \text{ if }\omega'_k+m=0\ [a_k],\\
                               0 \text{ otherwise,}
                             \end{cases}
      \]
and we can repeat word for word the end of the proof of Theorem~\ref{thm:eta_generic}.
\end{proof}

Now, what can we say about the measure $\nup$? Since the measurable dynamical system $(Y, \nup, S)$ is a factor of $(\Omega', \PP', T')$, it has zero entropy, and we cannot hope that this measure $\nup$ be the Bernoulli measure $\beta$. Moreover, observe that 
\[
 \PP'(\exists k: \omega'_k=0) \le \sum_{k\ge 1} \PP'(\omega'_k=0) = \Sigma.
\]
Therefore, with probability at least $1-\Sigma$, there is no $k$ such that $\omega'_k=0$, hence $(-1)^{\Delta(\omega')}=1$ with probability at least  $1-\Sigma$, which can be close to 1 if $\Sigma$ is small. However, the following results show that, as we approach the Möbius case, this measure $\nup$ for which $\pi$ is generic looks more and more like the Bernoulli measure~$\beta$. 

\begin{prop}
\label{prop:to1/2}
  As $\Sigma\to\infty$, we have the convergence 
  \[
    \PP'\left( (-1)^{\Delta}=1 \right) \longrightarrow \dfrac{1}{2}.
  \]
\end{prop}

\begin{proof}
 Since $\PP'\bigl(\omega'_k=0\bigr)=1/a_k$, denoting by $\EE'[\,\cdot\,]$ the expectation with respect to the probability $\PP'$, we have
\[
 \EE' \left[ (-1)^{\ind{\omega'_k=0}}\right] = 1-\dfrac{2}{a_k}.
\]
Then, by independence of the coordinates under $\PP'$, using the inequality $\ln(1+x)\le x$ for $-1<x<1$, we get
\begin{align*}
 \EE' \left[ (-1)^{\Delta}\right] 
    & =  \EE' \left[ (-1)^{\sum_{k\ge 1}\ind{\omega'_k=0}}\right] \\
    & = \prod_{k\ge 1}\EE' \left[ (-1)^{\ind{\omega'_k=0}}\right] \\
    & = \prod_{k\ge 1} \left( 1-\dfrac{2}{a_k} \right)\\
    & \le \exp(-2\Sigma).
\end{align*}
The proposition follows, since the above expectation is also equal to 
\[  
2\PP'\Bigl( (-1)^{\Delta}=1 \Bigr) -1.
\]
\end{proof}

\begin{theo}
\label{thm:Sigma_to_infty}
We have the weak convergence
  \[
    \nup\tend{\Sigma}{\infty}  \beta.
  \]
\end{theo}

\begin{proof}
For $m\ge1$, let us call an \emph{$m$-cylinder} any cylinder subset of $Y$ of the form
\[
  C = \left\{y\in Y: \forall 1\le j\le m,\ y_j=\alpha_j\right\}
\]
for a fixed choice of $(\alpha_1,\ldots,\alpha_m)\in\{-1,1\}^{m}$.
We have to prove that, for any $m\ge1$, and any $m$-cylinder $C$,
\begin{equation}
 \label{eq:weak-convergence}
 \lim_{\Sigma\to\infty}\nup\left(C \right) 
= \left( \dfrac{1}{2}\right)^{m}.
\end{equation}
We will prove this by induction on $m$, observing that Proposition~\ref{prop:to1/2} already gives the above convergence when $m=1$.
Assuming~\eqref{eq:weak-convergence} holds for $(m-1)$-cylinders, we will get it for $m$-cylinders by showing that the $\nup$-distribution of $y_m$ conditioned on $(y_1=\alpha_1,\ldots,y_{m-1}=\alpha_{m-1})$ can be made arbitrarily close to the uniform distribution on $\{-1,1\}$ if $\Sigma$ is large enough. But, for a random variable taking values in $\{-1,1\}$,  saying that its distribution is close to the uniform distribution on $\{-1,1\}$ amounts to saying that its expected value is close to 0. Hence, translated into the probability space $(\Omega',\PP')$, what we want to prove is equivalent to
\begin{equation}
 \label{eq:conditional_probability}
\lim_{\Sigma\to\infty} \EE' \left[ (-1)^{\Delta\circ T'^m} \ \left|\ (-1)^{\Delta\circ T'}=\alpha_1,\ldots,(-1)^{\Delta\circ T'^{m-1}}=\alpha_{m-1}   \right. \right] = 0.
\end{equation}

Recall that we assumed the validity of~\eqref{eq:weak-convergence} for $(m-1)$-cylinders. Hence we can take $\Sigma$ large enough so that 
\[
  \PP'\left(\ (-1)^{\Delta\circ T'}=\alpha_1,\ldots,(-1)^{\Delta\circ T'^{m-1}}=\alpha_{m-1}\right) \ge \dfrac{1}{2^{m}}.
\]
Then, \eqref{eq:conditional_probability} is equivalent to
\begin{equation}
 \label{eq:conditional_probability2}
  \EE' \left[ (-1)^{\Delta\circ T'^m} \ind{(-1)^{\Delta\circ T'}=\alpha_1,\ldots,(-1)^{\Delta\circ T'^{m-1}}=\alpha_{m-1}} \right] \tend{\Sigma}{\infty} 0.
\end{equation}

For $j\in\NNS$, we introduce the random subset $K_j=K_j(\omega')\subset \{1,2,3,\ldots\}$, defined as follows:
\[
 K_j \egdef\ \bigl\{k\ge1: \omega'_k + j = 0\ [a_k]\bigr\},
\]
so that $\Delta\circ T'^j=|K_j|$. Conditioning with respect to $(-1)^{\Delta\circ T'^j}$, $1\le j\le m-1$ is not convenient: It is more appropriate to condition on a more precise information, namely the knowledge of $K_1,\ldots,K_{m-1}$.

Observe that the information provided by $K_1,\ldots,K_{m-1}$ is precisely the following: For each $k\ge 1$ and each $1\le j\le m-1$, we know whether $\omega'_k+j=0\ [a_k]$ or not. Then, with respect to the conditional probability given  $K_1,\ldots,K_{m-1}$, we have the three following properties:
\begin{itemize}
  \item The coordinates $\omega'_k$, $k\ge1$ remain independent.
  \item If $k\in K_j$ for some $1\le j\le m-1$, since we know that $\omega'_k+j = 0\ [a_k]$, then either $j=m\ [a_k]$, and we have $\omega'_k+m = 0\ [a_k]$, or $j\neq m\ [a_k]$, and we have $\omega'_k+m \neq 0\ [a_k]$. In either case, we have
  \[
    \left| \EE'\left[ (-1)^{\ind{\omega'_k+m=0\ [a_k]}} \bigl| K_1,\ldots,K_{m-1} \right] \right| = 1.
  \]

  \item If $k\not \in K_1\cup\cdots\cup K_{m-1}$, then we know that $\omega'_k+j\neq 0\ [a_k]$ for each $1\le j\le m-1$. Hence $a_k>m-1$, and $\omega'_k$ is uniformly distributed on $\ZZ/a_k\ZZ\setminus\{a_k-1,\ldots,a_k-(m-1)\}$.  We then have 
  \begin{equation}
    \label{eq:K_m}
    \PP'\bigl( k\in K_m | K_1,\ldots,K_{m-1} \bigr) = \dfrac{1}{a_k-m+1},
  \end{equation}
and
  \[
    \EE'\left[ (-1)^{\ind{\omega'_k+m=0\ [a_k]}} \bigl| K_1,\ldots,K_{m-1} \right] = 1-\dfrac{2}{a_k-m+1}.
  \]
  \end{itemize}
By the same argument as in the proof of Proposition~\ref{prop:to1/2}, we get
\begin{align}
  \nonumber\left| \EE'\left[ (-1)^{\Delta\circ T'^m} \bigl| K_1,\ldots,K_{m-1} \right]  \right| 
  & =  \prod_{\stack{k\ge1}{k\not\in K_1\cup\cdots\cup K_{m-1}}} \left(1-\dfrac{2}{a_k-m+1}\right)  \\
  & \le \exp(-2F), 
  \label{eq:ineq_F}
\end{align}
where $F$ is the random variable defined by
\[
  F \egdef\ \sum_{\stack{k\ge1}{k\not\in K_1\cup\cdots\cup K_{m-1}}} \dfrac{1}{a_k-m+1}
\]
(the randomness of $F$ coming from the role of the random sets $K_j$).
Now, we have to prove that, if $\Sigma$ is large, then $F$ is large with probability $\PP'$ close to~1.

Fix some small number $\varepsilon>0$, and a number $M$ large enough so that $\frac{2}{M-1}<\varepsilon$. We have 
\begin{align*}
  F & \le \sum_{\stack{k\ge1}{a_k>m-1}} \dfrac{1}{a_k-m+1}\\
    & = \sum_{\stack{k\ge1}{m-1<a_k\le Mm}} \dfrac{1}{a_k-m+1} + \sum_{\stack{k\ge1}{a_k>Mm}} \dfrac{1}{a_k-m+1}.
\end{align*}
The first sum can be bounded by $\sum_{1\le \ell\le Mm}1/\ell$, which is less than $\Sigma/(M-1)$ if $\Sigma$ is large enough. In the second sum, we have $a_k-m+1\ge a_k(M-1)/M $, hence this second sum can be bounded by $\Sigma M/(M-1)$. We then get, if $\Sigma$ is large enough,
\begin{equation}
  \label{eq:bound_for_F}
  F\le \dfrac{M+1}{M-1} \Sigma \le (1+\varepsilon) \Sigma.
\end{equation}

On the other side, by~\eqref{eq:K_m} we can interpret $F$ as
\[
  F = \EE'\Bigl[ \,|K_m\setminus K_1\cup\cdots\cup K_{m-1}|\ \bigl|K_1,\ldots,K_{m-1} \Bigr].
\]
Observe that, if $k\in K_m \cap (K_1\cup\cdots\cup K_{m-1})$, then $\omega'_k+m=0\ [a_k]$, and there exists also $1\le j\le m-1$ with $\omega'_k+j=0\ [a_k]$. It follows that $j=m\ [a_k]$, and necessarily $a_k\le m$. Since the integers $a_k$ are distinct, we deduce that
\[
 \Bigl| K_m\cap (K_1\cup\cdots\cup K_{m-1}) \Bigr| \le m,
\]
and it follows that 
\[
  F \ge \EE'\Bigl[ \,|K_m|\ \bigl|K_1,\ldots,K_{m-1} \Bigr] - m.
\]
Therefore, provided $\Sigma$ is large enough,
\begin{equation}
  \label{eq:expectation_of_F}
  \EE'[F] \ge \EE'\bigl[ \,|K_m|\, \bigr] -m = \Sigma - m \ge (1-\varepsilon)\,\Sigma.
\end{equation}
Then, using~\eqref{eq:bound_for_F}, we get
\begin{align*}
  (1-\varepsilon)\,\Sigma &\le \EE'[F] \\
    & = \EE'\left[ F \ind{F\le \varepsilon\Sigma} \right] + \EE'\left[ F \ind{F > \varepsilon\Sigma} \right] \\
    & \le \varepsilon\,\Sigma + (1+\varepsilon)\,\Sigma\ \PP'(F>\varepsilon\Sigma),
\end{align*}
and finally
\begin{equation}
  \label{eq:F_large}
  \PP'(F>\varepsilon\Sigma) \ge \dfrac{1-2\varepsilon}{1+\varepsilon}.
\end{equation}

Now, let us come back to our goal, which is to establish~\eqref{eq:conditional_probability2}. 
If $\Sigma$ is large enough, we have
\begin{align*}
  &\left| \EE' \left[ (-1)^{\Delta\circ T'^m} \ind{(-1)^{\Delta\circ T'}=\alpha_1,\ldots,(-1)^{\Delta\circ T'^{m-1}}=\alpha_{m-1}} \right] \right| \\
  & = \left| \EE' \left[ \EE' \bigl[(-1)^{\Delta\circ T'^m} \bigl| K_1,\ldots,K_{m-1}\bigr] \ind{(-1)^{\Delta\circ T'}=\alpha_1,\ldots,(-1)^{\Delta\circ T'^{m-1}}=\alpha_{m-1}} \right] \right| \\
  & \le \EE'\left[ \, \bigl|\EE' \bigr[(-1)^{\Delta\circ T'^m} \bigl| K_1,\ldots,K_{m-1}\bigr] \bigr| \, \right]\\
  & \le \EE'\left[ \exp(-2F) \right] \qquad\text{by \eqref{eq:ineq_F}}\\
  & \le \exp(-2\varepsilon\Sigma) + \PP'(F\le\varepsilon\Sigma).
\end{align*}
By~\eqref{eq:F_large}, the second term can be made as small as we want by taking $\varepsilon$ small enough, then the first term can also be arbitrarily small by taking $\Sigma$ large enough.
\end{proof}

In view of Theorem~\ref{thm:Sigma_to_infty}, we are led to formulate the following conjecture which, by Theorem~\ref{thm:chowla and genericity}, would imply the validity of the Chowla conjecture.

\begin{conjecture}
  When $\Sigma=\infty$, the sequence $\pi$ is generic for $\beta$.
\end{conjecture}


\bibliography{mobius}

\end{document}